\documentclass[10pt]{article}
\usepackage{amssymb,latexsym,amscd, amsthm, amsmath}

\newtheorem{Thm}{Theorem}[section]
\newtheorem{Cor}[Thm]{Corollary}
\newtheorem{Prop}[Thm]{Proposition}
\newtheorem{Lem}[Thm]{Lemma}
\newtheorem{Def}[Thm]{Definition}
\newtheorem{Ex}[Thm]{Example}
\newtheorem{Rmk}[Thm]{Remark}

\date{}

\newcommand{\TT}{{\mathcal{T}}}

\newcommand{\RR}{{\mathbb{R}}}
\newcommand{\ZZ}{{\mathbb{Z}}}
\newcommand{\NN}{{\mathbb{N}}}
\newcommand{\QQ}{{\mathbb{Q}}}
\newcommand{\GT}{{\mathcal{G}_{2,n}}}

\newcommand{\GM}{{\mathcal{G}_{m,n}}}
\newcommand{\TTT}{{\mathcal{T}_{3,n}}}
\newcommand{\TM}{{\mathcal{T}_{m,n}}}
\newcommand{\TN}{{\mathcal{T}_{n}}}
\newcommand{\CM}{{\mathcal{C}_{m}}}
\newcommand{\val}{{\text{val}}}
\newcommand{\trop}{{\text{trop}}}

\begin{document}

\title{A tropical interpretation of $m-$dissimilarity maps}
\author{Cristiano Bocci \footnote{I.T.I.S. ``A. Avogadro", Via case Nuove 27, 53021 Abbadia San Salvatore (SI), Italy, email:
cristiano.bocci@gmail.com .} \hspace{1mm} and Filip Cools \footnote{K.U.Leuven,
Department of Mathematics, Celestijnenlaan 200B, B-3001 Leuven,
Belgium, email: Filip.Cools@wis.kuleuven.be .}}
\maketitle {\footnotesize \emph{\textbf{Abstract.} Let $T$ be a weighted tree with $n$ numbered leaves and let $D = (D(i,j))_{i,j}$ be its distance matrix, so $D(i,j)$ is the distance between the leaves $i$ and $j$. If $m$ is an integer satisfying $2\leq m\leq n$, we prove a tropical formula to compute the $m$-dissimilarity map of $T$ (i.e. the weights of the subtrees of $T$ with $m$ leaves), given $D$. For $m=3$, we present a tropical description of the set of $m$-dissimilarity maps of trees. For $m=4$, a partial result is given. }}
\\ ${}$
\maketitle


\section{Introduction}

Let $D$ be a matrix whose rows and columns are indexed by a set $X$.
We assume that $D$ is symmetric and has zero entries on the main
diagonal. In phylogenetics, these kind of matrices are called {\it
dissimilarity matrices} . Usually, we take $X=[n]:=\{1, 2, \dots,
n\}$. Hence a dissimilarity matrix $D$ can also be seen as a map
$D:[n]^2\to \RR$, with $D(i,j)=D(j,i)$ and $D(i,i)=0$ for each $i,j\in[n]$. 

A {\it metric} is a non-negative dissimilarity matrix which satisfies the
triangle inequality $D(i,j)\leq D(i,k)+D(k,j)$ for all $i,j,k\in X$.
 
We say that $D$ has a {\it graph realization} if there is a weighted graph 
(so a non-negative weight is assigned to each edge) 
whose node set contains $X$ and such that the distance (i.e. the length of
the shortest path) between nodes $i,j \in X$ is exactly $D(i,j)$. A {\it distance
matrix} is a non-negative dissimilarity matrix that has a graph
realization. In \cite{Chu, Hak}, one can find some results on these kind of matrices.

In the case the graph is a tree and $X$ corresponds to
the set of leaves, $D$ is called a {\it tree metric}. This case has
been studied intensively and is well understood. The main
result is the following (see \cite{Bun} or \cite[Theorem 2.36]{PaSt}).

\begin{Thm}[Tree Metric Theorem] \label{thm tree metric}
Let $D$ be a non-negative dissimilarity matrix on $[n]$. Then $D$ is
a tree metric on $[n]$ if and only if, for every four (not
necessarily distinct) elements $i,j,k,l\in [n]$, the maximum
of the three numbers $D(i,j)+D(k,l)$, $D(i,k)+D(j,l)$ and
$D(i,l)+D(j,k)$ is attained at least twice. Moreover, the tree $T$
with leaves $[n]$ that realizes $D$ is unique.
\end{Thm}

The condition of the theorem is called the {\it four-point
condition}. It is a necessary and sufficient condition on a matrix
to be realized by a tree.

Tree metrics on $n$ leaves are parameterized by the {\it space of trees} $\TN \subset \RR^{n \choose 2}$. 
The following result gives us a description of $\TN$ (see \cite{BHV}). 

\begin{Thm}
The space of trees $\TN$ is the union of $(2n-5)!!=1.3.5\ldots.(2n-5)$ orthants isomorphic to $\RR^{2n-3}_{\geq 0}$.
More precisely, $\TN$ is a simplicial fan of pure dimension $2n-3$ in $\RR^{n \choose 2}$.
\end{Thm}

We can consider a generalisation of the concept of dissimilarity matrix. Let $m\leq n$ be an integer. A map $D:[n]^m\to \RR$ 
is called an {\it $m$-dissimilarity map} if $$D(i_1,\dots,i_m)=D(i_{\pi(1)},\dots,i_{\pi(m)})$$ for all permutations $\pi \in
S_m$ and $D(i_1, i_2,\ldots, i_m)=0$ if the numbers $i_1,\ldots,i_m$ are not pairwise distinct. 

We say that $D$ is realized by a tree $T$ if the leaf set of $T$ is $[n]$ and if for each $m$-subset $V=\{i_1,\ldots,i_m\}\subset [n]$,
the weight of the smallest subtree of $T$ containing $V$ is equal to $D(i_1,\ldots,i_m)$.  
An important result on $m$-dissimilarity maps of trees is given in \cite{PaSp}.

\begin{Thm}\label{thm pachter-speyer}
Let $T$ be a tree with $n$ leaves and no vertices of degree 2. Let
$m\geq 3$ be an integer. If $n\geq 2m-1$, then $T$ is uniquely determined by
its $m$-dissimilarity map $D$. If $n=2m-2$, this is not true.
\end{Thm}

In this paper, we give a description of a map $\phi^{(m)}:\RR^{n\choose 2}\to\RR^{n\choose 3}$, sending the distance matrix of a tree $T$ to its corresponding $m$-dissimilarity map (see Theorem \ref{thm description phi} in Section \ref{section phi}). In Section \ref{section m=3}, we investigate the case $m=3$. In particular, we show that $\phi^{(3)}(\TN)$ is equal to the intersection of the tropical Grassmannian $\mathcal{G}_{3,n}$ with a linear space (see Theorem \ref{thm m=3}). In Section \ref{section m=4}, we give a partial result on the case $m=4$. An introduction to tropical gemetry is given in Section \ref{section tropical}.

To finish this section, we describe the relation with Phylogenetics. A classical problem in computational biology is to construct a phylogenetic tree from a sequence alignment of $n$ species

\begin{center}
\begin{tabular}{ll}
Species 1        &\texttt{\small
ACAATGTCATTAGCGATACGTAGGTACGATGC...} \\
Species 2   &\texttt{\small
ACGTTGTCAATAGAGATTTTGGATGAACGATA...} \\
Species 3      &\texttt{\small
ACGTAGTCATTACACATTCTGGATTAACGTTA...} \\
Species 4    &\texttt{\small
GCACAGTCAGTAGAAGCTATGGTACATCGATC...} \\
$\qquad \vdots$ & $\quad \vdots \qquad \vdots \qquad \vdots \qquad \vdots \qquad \vdots \qquad \vdots \qquad \vdots$\\
Species n    &\texttt{\small
GAACTGTCAGTAGAAGCGAGTGTACATTCGTT...} 
\end{tabular}
\end{center}

The main technique to select a tree model is computing the \textit{maximum likelihood estimate} (MLE) for each of the $(2n-5)!!$ trees. Unluckily, all the MLE computations are very difficult, even for a single tree, and this approach requires examining all exponentially many trees.

A popular way to avoid this problem is the so-called \textit{distance based approach}, where one collapses the data to a dissimilarity matrix and obtains a tree via a projection onto tree space $\TN$ (by using the neighbor-joining algorithm). In fact, for such sequence data, computational biologists infer the distance between any two taxa. Thus, an interesting problem of phylogenetics concerns the construction of a weighted tree which represents this distance matrix, provided such a tree exists. 

More general, we may think of an $m$-dissimilarity map as a measure of how dissimilar each subset of $m$ species is. As a generalization of the previous problem, we can search for a weighted tree such that the $m$-subtree weights represent the entries of the $m$-dissimilarity map.
This problem has some natural relevance in Phylogenetics. Indeed, for example, it can be more reliable statistically to estimate the triple weights $D(i,j,k)$ rather than the pairwise distances $D(i,j)$ (\cite{PaSp}, \cite{PaSt}).

\section{Tropical geometry} \label{section tropical}

The tropical semiring $(\RR\cup \{-\infty\}, \oplus, \otimes )$ is
the set of real numbers completed with $-\infty$, equiped with two binary operations: 
the tropical sum is the maximum of two numbers and the tropical multiplication is
the ordinary sum.

Tropical monomials $x_1^{a_1}\cdots x_k^{a_k}$ represent ordinary
linear forms $\sum_{i=1}^k a_i x_i$ and tropical polynomials
\begin{equation}  \label{troppoly}
\bigoplus_{a \in A} \lambda_a \otimes x_1^{a_1}\otimes \cdots
\otimes x_k^{a_k},
\end{equation}
with $A \subset \NN^k$ finite and
$\lambda_a \in \RR$, represent piecewise-linear convex functions 
\begin{equation} \label{function}
F:\RR^k
\to \RR:(x_1,\ldots,x_k)\mapsto \max_{a\in A}\{\lambda_a+\sum_{i=1}^k a_i x_i\}.
\end{equation}

Now let $K$ be the {\it field of Puiseux series}, i.e. the field of formal power series 
$a=\sum_{q\in\mathbb{Q}}\,a_q t^q$ in the variable $t$ such that the set $Q_a=\{q\in\mathbb{Q}\,|\,a_q\neq 0\}$ 
is bounded below and has a finte set of denominators. For such an $a$, the infimum of $Q_a$ is equal to the minimum and we
call it the {\it valuation} $\val(a)$ of $a$.

A polynomial $$f(x_1,\cdots,x_k)=\sum_{a\in A}\,g_a(t)x_1^{a_1}\cdots x_k^{a_k}\in K[X]$$ 
gives rise to the tropical polynomial in \eqref{troppoly}, where $\lambda_a=-\val(g_a(t))$. We denote this tropical polynomial by $\trop(f)$.

We define the {\it tropical hypersurface} $\TT(F)=\TT(\trop(f))$ as the corner locus of
the function $F$ in \eqref{function}, i.e. the set of $x=(x_1,\dots,x_k)\in \RR^k$
such that the maximum of the collection of numbers
$$\left\{\sum_{i=1}^k a_i x_i+\lambda_a\right\}_{a\in\mathcal{A}}$$ is attained at
least twice.

\begin{Thm} \label{Thm def tropical variety}
If $I\subset K[x_1,\dots,x_n]$ is an ideal, the following two subsets of $\RR^k$ coincide:
\begin{enumerate}
\item the intersection of all tropical hypersurfaces $\TT(\trop(f))$ with $f\in I$;
\item the closure in $\RR^k$ of the set $$\{(-\val(y_1),\ldots,-\val(y_k))\,|\,(y_1,\ldots,y_k)\in V(I)\}\subset \QQ^k.$$ 
\end{enumerate}
\end{Thm}
\begin{proof} See \cite[Theorem 2.1]{SS1}.
\end{proof}

For an ideal $I\subset K[x_1,\ldots,x_k]$, we denote by $\TT(I)\subset \RR^k$ the set 
mentioned in Theorem \ref{Thm def tropical variety}. It is called the {\it tropical variety} of the ideal $I$.

\begin{Def} If $\TT(I)\subset \RR^k$ is a tropical variety, we say that $\{f_1,\ldots,f_r\}$ is a tropical basis of $\TT(I)$ if and only if $I=\langle f_1,\ldots,f_r\rangle$ and $$\TT(I)=\TT(\trop(f_1))\cap \cdots \cap \TT(\trop(f_r)).$$  
\end{Def}

\begin{Rmk} {\fontshape{n}\selectfont
In general, a set of generators of an ideal $I$ is not a tropical basis for $\TT(I)$. Of course, the singleton $\{f\}$ is a tropical basis for the tropical hypersurface $\TT(\trop(f))$. 
}\end{Rmk}

We are mainly interested in the tropical variety $\TT(I_{m,n})$,
where $I_{m,n}$ is the ideal of the {\it Grassmannian} $G(m,n)\subset
\RR^{n \choose m}$. To be more precise, we fix a polynomial ring $$\ZZ[x] = \ZZ[x_{i_1i_2\cdots i_d}\,|\,1\leq i_1<i_2 <\cdots <i_m\leq n]$$ in
${n \choose m}$ variables with integer coefficients. The Pl\"ucker ideal $I_{m,n}$ is the prime ideal in $\ZZ[x]$,
consisting of the algebraic relations among the determinants of the $(m\times m)$-minors of any $(m\times n)$-matrix with entries in a
commutative ring. It is well-known that $I_{m,n}$ is generated by quadrics (see for example \cite{Stu}).

The affine variety defined by $I_{m,n}$ is the Grassmannian $G(m,n)\subset \RR^{n \choose m}$, which parameterizes all
$m$-dimensional linear subspaces of an $n$-dimensional vector space. It has dimension $(n-m)m+1$.

\begin{Def}
The tropical variety $\TT(I_{m,n})$ is called a tropical Grassmannian and is denoted by $\GM$.
\end{Def}

\begin{Thm}
The tropical Grassmannian $\GM$ is a polyhedral fan in $\RR^{n
\choose m}$. Each of its maximal cones has the same dimension,
namely $(n-m)m+1$.
\end{Thm}
\begin{proof}
See \cite[Corollary 3.1.]{SS1}.
\end{proof}

Now we are going to fix our attention on the case $m=2$.

\begin{Ex}[$m=2$ and $n=4$] {\fontshape{n}\selectfont
The smallest non-zero Pl\"ucker ideal is the principal ideal
$I_{2,4}=(x_{12}x_{34}-x_{13}x_{24}+x_{14}x_{23})$. Thus
${\mathcal{G}}_{2,4}$ is a fan with three five-dimensional cones
$\RR^4\times\RR_{\leq 0}$ glued along $\RR^4$.
} \end{Ex}

\begin{Thm}\label{thm generators of I_{2,n}}
The ideal $I_{2,n}$ is generated by the quadratic polynomials
\begin{equation}\label{PLU}
p_{ijkl}:=\underline{x_{ik}x_{jl}}-x_{ij}x_{kl}-x_{il}x_{jk} \qquad (1\leq
i<j<k<l \leq n).
\end{equation}
These polynomials form the reduced Gr\"obner basis if the underlined terms are leading.
\end{Thm}
\begin{proof}
See \cite[Theorem 3.1.7 and Proposition 3.7.4]{Stu}.
\end{proof}

For each quadruple $\{i,j,k,l\}\subset\{1,2,\dots, n\}$, we consider
the tropical polynomial
$$\trop(p_{ijkl}) = (x_{ij}\otimes x_{kl}) \oplus (x_{ik}\otimes x_{jl})
\oplus (x_{il}\otimes x_{jk}).$$
This polynomial defines a tropical hypersurface $\TT(\trop(p_{ijkl}))$. 
It turns out that the tropical Grassmannian $\GT$ is the intersection of these ${n \choose 4}$ hypersurfaces, so 
the quadrics $p_{ijkl}$ forms a tropical basis for $I_{2,n}$ (see \cite{SS1}).

Let $D$ be an dissimilarity matrix on $[n]$ and $\{i,j,k,l\}\subset [n]$. The maximum of the
three numbers $D(i,j)+D(k,l)$, $D(i,k)+D(j,l)$ and $D(i,l)+D(j,k)$
is attained at least twice if and only if $D\in \TT(\trop(p_{ijkl}))$. Thus Theorem \ref{thm tree metric} implies 
that a metric $D$ on $[n]$ is a tree metric if and only if $D$ belongs to $\TN$. In particular,
one has the following result.

\begin{Thm}
The space of trees $\TN$ is the tropical Grassmannian $\GT$.
\end{Thm}
\begin{proof}
See \cite[Theorem 4.2]{SS1} or the arguments above.
\end{proof}

Now we come back to the general case (so the case where $m\leq n$ is arbitrary). The ideal $I_{m,n}$ is
generated by quadratic polynomials, known as the Pl\"ucker
relations. Among these are the three-term Pl\"ucker relations
$$p_{R,ijkl}:=x_{Rik} x_{Rjl}- x_{Rij} x_{Rkl} - x_{Ril} x_{Rjk},$$
which are closely related to (\ref{PLU}). Hereby $R$ is any
$(m-2)$-subset of $[n]$ and $i, j, k, l \in [n]\setminus R$.

\begin{Def}
The three-term tropical Grassmannian $\TM$ is the intersection
$$
\TM:= \bigcap_{R,i,j,k,l} \TT(\trop(p_{R,ijkl}))\quad \subset \RR^{n \choose m}.
$$
\end{Def}

In general, the three-term Pl\"ucker relations do not generate
$I_{m,n}$. If $m=2$, then $S=\emptyset$ and ${\mathcal{T}}_{2,n}=\GT$. For $m\geq 3$, the tropical Grassmannian
$\GM$ is contained in $\TM$. This containment is proper for $n\geq m+4$. 

\section{A description on the $m$-subtree weight map} \label{section phi}

In this section, we are going to give an explicit description of a
map $$\phi^{(m)}:\RR^{n \choose 2} \to \RR^{n \choose m},$$ sending the
dissimilarity matrix $D$ of a tree $T$ to its $m$-dissimilarity map.

Let $\prec$ be the order relation on $\mathbb{N}^{\infty}$ defined
as follows. We have $$(a_1,a_2,a_3,\ldots)\prec
(b_1,b_2,b_3,\ldots)$$ if and only if there exists an $n\in \NN$
such that $a_i=b_i$ for all $i<n$ and $a_n<b_n$.

Let $T$ be a tree with $n$ leaves. Let $r$ be an inner node of $T$
and consider $T$ as a rooted tree (with root $r$). Let $\mathcal{N}$
be the set of nodes of $T$. In particular, the set of leaves
$[n]=\{1,\ldots,n\}$ is contained in $\mathcal{N}$.

\begin{Lem}There exists a map $\alpha:\mathcal{N}\to
\NN^{\infty}$ such that the following properties hold:
\begin{enumerate}
\item $\alpha$ is injective.
\item If $n\in \mathcal{N}$ is an ancestor of $m\in \mathcal{N}$, we
have $\alpha(m)\succ \alpha(n)$. So the root $r$ of $T$ gives rise
to the minimum of $\{\alpha(n)|n\in \mathcal{N}\}$.
\item If $n_1,n_2\in \mathcal{N}$ with $n_2$ not a descendant nor an ancestor of $n_1$, $m_1\in \mathcal{N}$ a descendant of $n_1$ and $m_2\in \mathcal{N}$ a descendant of $n_2$, we have $\alpha(m_1)\prec \alpha(m_2)$ if and only if
$\alpha(n_1)\prec \alpha(n_2)$.
\end{enumerate}
\end{Lem}
\begin{proof}
We will define $\alpha$ inductively. Take
$\alpha(r)=(0,0,0,\ldots)$. For the induction step, if
$\alpha(n)=(a_1,\ldots,a_s,0,0,\ldots)$ is defined for some $n\in
\mathcal{N}$ with $a_s\neq 0$ and if $m_1,\ldots,m_t$ are the
children of $n$, take $\alpha(n_i)=(a_1,\ldots,a_s,i,0,\ldots)$.
Note that all the properties hold and that the depth of $n\in
\mathcal{N}$ in $T$ is equal to the number of non-zero entries in
$\alpha(n)$.
\end{proof}

We say that the leaves of $T$ are {\it well-numbered} if and only if
$\alpha(i)\prec \alpha(j)$ for all $i<j$.

A permutation $\sigma\in S_m$ of $\{1,\ldots,m\}$ is called {\it cyclic} if and only if
the decomposition of $\sigma$ into a product of disjoint cycles
consists of only one cycle of order $m$. Denote the set of cyclic
permutations in $\mathcal{S}_m$ by $\mathcal{C}_m$. Note that
$\sigma^m=Id$ if $\sigma \in \mathcal{C}_m$.

\begin{Thm} \label{thm description phi}
Let $n$ and $m$ be integers such that $n>m\geq 2$. Let
$$\phi^{(m)}:\RR^{n \choose 2}\to \RR^{n \choose m}:
X=(X_{i,j})\mapsto (X_{i_1,\ldots,i_m})$$ be the map with
$$
X_{i_1,\ldots,i_m}=\frac12 \cdot \min_{\sigma \in \mathcal{C}_m}
\{X_{i_1,i_{\sigma(1)}}+X_{i_{\sigma(1)},i_{\sigma^2(1)}}+\ldots+X_{i_{\sigma^{m-1}(1)},i_{\sigma^m(1)}}\}.
$$
If $D\in \mathcal{G}_{2,n}\subset \mathbb{R}^{n \choose 2}$ is the dissimilarity matrix of an $n$-tree $T$, then
the $m$-dissimilarity map of $T$ is equal to $\phi^{(m)}(D)$. So the set of $m$-dissimilarity maps of $n$-trees is equal to $\phi^{(m)}(\mathcal{G}_{2,n})$.
\end{Thm}
\begin{proof} Write
$$f(X;\sigma;i_1,\ldots,i_m)=X_{i_1,i_{\sigma(1)}}+X_{i_{\sigma(1)},i_{\sigma^2(1)}}+
\ldots+X_{i_{\sigma^{m-1}(1)},i_{\sigma^m(1)}}.$$ Note that
$$f(X;\sigma;i_{\pi(1)},\ldots,i_{\pi(m)})=f(X;\pi\sigma\pi^{-1};i_1,\ldots,i_m)$$
for all $\pi\in \mathcal{S}_m$, hence
\begin{equation} \label{indep. of perm.}
\min_{\sigma \in \mathcal{C}_m}
\{f(X;\sigma;i_{\pi(1)},\ldots,i_{\pi(m)})\}=\min_{\sigma \in
\mathcal{C}_m} \{f(X;\sigma;i_1,\ldots,i_m)\}.
\end{equation}

We have to prove that the weight $D(i_1,\ldots,i_m)$ of the smallest
subtree $T'$ of $T$ containing the leaves $i_1,\ldots,i_m$ is equal
to $\frac12 \cdot\min_{\sigma \in \mathcal{C}_m}
\{f(D;\sigma;i_1,\ldots,i_m)\}$. It is enough to prove this for
$i_1=1,\ldots,i_m=m$ (the general case is proved completely
analogously). By equation (\ref{indep. of perm.}), we may also
assume the leaves of $T'$ are well-numbered.

Let $e=(x,y)$ be an edge of $T'$ with $y$ a child of $x$. We claim
that for all $\sigma\in \mathcal{C}_m$, the weight $w(e)$ of $e$
is taken into account in at least two of the $m$ terms of
$$f(D;\sigma;1,\ldots,m)=D(1,\sigma(1))+D(\sigma(1),\sigma^2(1))+\ldots+D(\sigma^{m-1}(1),1)$$
and in exactly two of the summands of
$$f(D;\tau;1,\ldots,m)=D(1,2)+D(2,3)+\ldots+D(m,1),$$ where
$$\tau=\left(\begin{matrix} 1 & 2 & \ldots & m-1& m\\ 2& 3& \ldots&
m& 1\end{matrix}\right)\in \mathcal{C}_m.$$ Using this claim, we immediately see
$$D(i_1,\ldots,i_m)=\frac12 \cdot f(D;\tau;1,\ldots,m)=\frac12 \cdot\min_{\sigma \in
\mathcal{C}_m} \{f(D;\sigma;1,\ldots,m)\}.$$

To finish this theorem, we only need to prove the claim. Consider
the split of $T'$ induced by $e$ and let $T''$ be the component of
the split containing $y$ (hence $T''$ is the maximal subtree of $T'$
containing $y$ but not $x$). Denote the set of leaves of $T''$ by
$L''$. We may assume $1\in L''$ (the case $1\not\in L''$ is
analogous). Note that in this case $L''$ is of the form
$\{1,\ldots,s\}$ for some $s<m$.

The weight of $e$ is taken into account in the term $D(i,j)$ (i.e.
the path between the leaves $i$ and $j$ of $T''$ passes $e$) if and
only if $i\in L''$ and $j\not\in L''$ or vice versa. Thus $w(e)$
is only counted in the two terms $D(s,s+1)$ and $D(m,1)$ of
$f(D;\tau;1,\ldots,m)$.

So it is enough to show that there exists a $t\in\{0\ldots,m-1\}$
such that $\sigma^t(1)\in L''$ and $\sigma^{t+1}(1)\not\in L''$ (the
other case is proved analogously). If we assume this is not the case
(so $\sigma^t(1)\in L''$ implies $\sigma^{t+1}(1)\in L''$), we get
$L''=\{1,\ldots,m\}$, a contradiction.
\end{proof}

\begin{Cor} 
If $D\in \mathcal{G}_{2,n}\subset \mathbb{R}^{n \choose 2}$, we have that $D(i_1,\ldots,i_m)$ is equal to  
$$\left( \bigoplus_{\sigma \in \CM}
\left(D(i_1,i_{\sigma(1)})\otimes D(i_{\sigma(1)},i_{\sigma^2(1)})\otimes \cdots
\otimes D(i_{\sigma^{m-1}(1)}, i_{\sigma^m(1)})\right)^{-1}\right)^{-\frac{1}{2}}.$$
\end{Cor}

\begin{Rmk} {\fontshape{n}\selectfont
In each component $D(i_1,\dots,i_m)$, the minimum is attained at
least twice. Indeed, assume the minimum is attained for $\sigma=\tau$. Since
$$f(D;\tau;i_1,\ldots,i_m)=f(D;\tau^{-1};i_1,\ldots,i_m),$$ the
minimum is also attained for $\sigma=\tau^{-1}$. Note that this
could be useful for computations, since it permits us to consider
only $\frac{|{\mathcal{C}}_m|}{2}$ permutations. Furthermore, 
if $\{i_j,i_k\}$ is a cherry of $T'$, the minimum is 
also attained for $\sigma=(j k) \circ \tau \circ (j k)$, whereby $(j k)$ is the 
transposition in $\mathcal{S}_m$ switching $j$ and $k$. 
}\end{Rmk}
 
\begin{Rmk} {\fontshape{n}\selectfont
The map $\phi^{(m)}$ is not injective on the whole domain $\RR^{n
\choose 2}$. For example, consider $D,D' \in \RR^{n \choose
2}$, whereby $D(i,j)=1$ for all $1\leq i<j\leq n$ and $D'$ only differs
from $D$ in the last coordinates, with $D'(n-1,n)=2$. Clearly, one has $D \in
\GT$, $D' \not\in \GT$ and $\phi^{(m)}(D)=\phi^{(m)}(D')$. However, 
Theorem \ref{thm pachter-speyer} implies that the restriction of $\phi^{(m)}$ to $\GT$ is injective if $n\geq 2m-1$.
}\end{Rmk}

\begin{Prop} \label{wm}
$\phi^{(m)}(\GT) \subseteq \TM \cap \phi^{(m)}(\RR^{n\choose 2})$
\end{Prop}
\begin{proof}
The inclusion $\phi^{(m)}(\GT) \subset \phi^{(m)}(\RR^{n\choose 2})$ is obvious, 
while $\phi^{(m)}(\GT) \subset \TM $ follows from \cite{PaSp}. For sake of completeness, we include the proof in this paper.

Consider a tree $T$ with leaf set $[n]$ and distance matrix $D$. Let $R$ be an $(m-2)$-subset of $[n]$ and $i,j,k,l\in [n]\setminus R$. We have to prove that $$\phi^{(m)}(D)\in \TT(\trop(p_{R,ijkl})).$$ 

Let $[R]$ be the smallest subtree of $T$ containing the leaves in $R$ and let $T'$ be the tree obtained from $T$ by contracting $[R]$ to a point. Denote by $i',j'$, etc. the images of respectively $i,j$, etc. in $T'$. Note that $R'$ is a leaf of $T'$. We have $$D(R,i,j)=D'(R',i',j')+D(R),$$ hence $\phi^{(m)}(D)\in \TT(\trop(p_{R,ijkl}))$ if and only if $\phi^{(3)}(D')\in \TT(\trop(p_{R',i'j'k'l'}))$, where $D'$ is the distance matrix of $T'$. 

Now Remark \ref{remark phi for m=3} below implies $$D'(R',i',j')=\frac12(D'(i',j')+D'(i',R')+D'(j',R')),$$ so $\phi^{(3)}(D')\in \TT(\trop(p_{R',i'j'k'l'}))$ if and only if $D'\in \TT(\trop(p_{i'j'k'l'}))$. Hence the statement follows from Theorem \ref{thm tree metric}.  
\end{proof}

\section{The $3$-dissimilarity maps of trees} \label{section m=3}

Denote the coordinates of $\RR^{n \choose 2}$ by $X(i,j)$ (here we
index over all integers $i,j$ with $1\leq i<j\leq n$) and
the coordinates of $\RR^{n \choose 3}$ by $X(i,j,k)$ (here we index
over all integers $i,j,k$ with $1\leq i<j<k\leq n$). Recall
that if $D\in \mathcal{G}_{2,n}$ is a tree, $D(i,j)$ is the distance between leaf i and leaf j.

\begin{Rmk} \label{remark phi for m=3}
Since $\mathcal{C}_3=\{\sigma_1,\sigma_2\}$ with $$\sigma_1=\left(\begin{matrix} 1 & 2 & 3\\ 2& 3& 1\end{matrix}\right)\quad \text{and} \quad
\sigma_2=(\sigma_1)^{-1}=\left(\begin{matrix} 1 & 2 & 3\\ 3& 1& 2\end{matrix}\right),$$ the map $\phi^{(3)}$ sends $X=(X(i,j))_{i,j}$ to $(X(i,j,k))_{i,j,k}$ with $$X(i,j,k)=\frac12 \cdot(X(i,j)+X(i,k)+X(j,k)).$$ So if $D\in \mathcal{G}_{2,n}$, the $3$-subtree weights of the tree $D$ are given by $D(i,j,k)=\frac12 \cdot(D(i,j)+D(i,k)+D(j,k))$.
\end{Rmk}

The following results states that for the case $m=3$ the equality holds in Proposition \ref{wm} if $n\geq 5$.

\begin{Prop} \label{Prop equality for m=3} If $n\geq 5$, we have $\phi^{(3)}(\GT)=\mathcal{T}_{3,n}\cap \phi^{(3)}(\RR^{n\choose 2})$
\end{Prop}
\begin{proof}
By Proposition \ref{wm}, it is enough to show
that for a general point $P\in \phi^{(3)}(\RR^{n \choose 2})\cap \TTT$,
there exists a point $D\in \GT$ such that $\phi^{(3)}(D)=P$.
Since $P\in \phi^{(3)}(\RR^{n \choose 2})$, there exists a point
$D \in \RR^{n \choose 2}$ such that $\phi^{(3)}(D)=P$. It suffices to
prove that $D\in \GT$. In order to do this, we show that in each triplet
$$\{D(i,j)+D(k,l),D(i,k)+D(j,l),D(i,k)+D(j,k)\},$$ the maximum is attained at least
twice. Fix $S\in [n]\setminus \{i,j,k,l\}$ ($n\geq 5$). Since $P\in \TTT$, in the triplet 
$$\{P(S,i,j)+P(S,k,l),P(S,i,k)+P(S,j,l),P(S,i,l)+P(S,j,k)\},$$ 
the maximum is attained at least twice. Note that
\begin{equation*}
\begin{split}
P(S,i,j)+P(S,k,l) &= \frac12(C+D(i,j)+D(k,l)),\\
P(S,i,k)+P(S,j,l) &= \frac12(C+D(i,k)+D(j,l)),\\
P(S,i,l)+P(S,j,k) &= \frac12(C+D(i,k)+D(j,k)),
\end{split}
\end{equation*}
where $C=D(S,i)+D(S,j)+D(S,k)+D(S,l)$. Hence the maximum in
$\{D(i,j)+D(k,l),D(i,k)+D(j,l),D(i,k)+D(j,k)\}$ is also attained at least twice, thus $D\in \GT$
and $P \in \phi^{(3)}(\GT)$.
\end{proof}

For the proof of the proposition below, we need an extra definition. 

\begin{Def}
An ultrametric $D$ on $[n]$ is a metric  which satisfies the following strengthened version of the triangle inequality:
$$\forall i,j,k\in [n]\,:\,D(i,j)\leq \max\{D(i,k),D(j,k)\}.$$ Equivalently, at least two of the three terms $D(i,j), D(i,k), D(j,k)$ are the same. 
\end{Def}

\begin{Rmk} {\fontshape{n}\selectfont
In general, the dissimilarity matrix $D$ of a tree $T$ is not an ultrametric. In case $D\in\GT$ is an ultrametric, we can realize $D$ by an {\it equidistant} tree, i.e. a rooted tree such that the distance $F$ between the root and each leaf is equal. In particular, $2F=\max\{D(i,j)\,|\,i,j\in X \text{ and }i\neq j\}$.  
} \end{Rmk}

\begin{Prop} \label{Thm inclusion} $\phi^{(3)}(\GT)\subset \mathcal{G}_{3,n}$
\end{Prop}
\begin{proof}
Let $T$ be a tree with $3$-dissimilarity map $$P=(D(i,j,k))_{i,j,k}=\phi^{(3)}((D(i,j)_{i,j})\in \phi^{(3)}(\GT)\subset \RR^{n\choose 3}.$$ 
If $M\in K^{3\times n}$, we denote the $(3\times 3)$-minor with columns $i,j,k$ by $M(i,j,k)$.  
By Theorem \ref{Thm def tropical variety}, $\mathcal{G}_{3,n}$ is the closure in $\RR^{n\choose 3}$ of the set $$S:=\{(-\val(\det(M(i,j,k))))_{i,j,k}\,|\,M\in K^{3\times n}\}\subset \QQ^{n\choose 3}.$$ Assume first that all the edges of $T$ have rational weights, a fortiori $P\in \QQ^{n\choose 3}$. We are going to show there exists a matrix $M\in K^{3\times n}$ such that $$D(i,j,k)=-\val(\det(M(i,j,k))).$$

Fix a rational number $E$ with $E\geq D(i,n)$ for all $i\in\{1,\ldots,n-1\}$ and define a new metric $D'$ by $$D'(i,j)=2E+D(i,j)-D(i,n)-D(j,n)$$ for all different $i,j\in[n]$ (in particular, $D'(i,n)=2E$ for $i\neq n$). Note that $D'\in \GT$ and that $D'$ an ultrametric on $\{1,\ldots,n-1\}$, so it can be realized by an equidistant $(n-1)$-tree $T''$ with root $r$. Each edge $e$ of $T''$ has a well-defined height $h(e)$, which is the distance from the top node of $e$ to each leaf below $e$. Pick a random rational number $a(e)$ and associate the label $a(e)t^{2h(e)}$ to $e$. If $i\in\{1,\ldots,n-1\}$ is a leaf of $T''$, define the polynomial $x_i(t)$ by adding the labels of all edges between $r$ and $i$. It is easy to see that $D'(i,j)=\deg(x_j(t)-x_i(t))$ for all $i,j\in\{1,\ldots,n-1\}$.    

Denote the distance from $r$ to each edge by $F$. Since $$2F=\max\{D'(i,j)\,|\,1\leq i< j\leq n-1\}<2E,$$ we have $F<E$. The metric $D'$ on $[n]$ can be realized by a tree $T'$, where $T'$ is the tree obtained from $T''$ by adding the leave $n$ together with an edge $(r,n)$ of length $2E-F$.     
If we define $x_n(t)=t^{2E}$, we get that $D'(i,j)=\deg(x_j(t)-x_i(t))$ for all $i,j\in[n]$. 

Now consider the matrix 
$$M'=\begin{bmatrix} 1&1&1&\ldots&1 \\ x_1(t)&x_2(t)&x_3(t)& \ldots &x_n(t)\\ x_1(t)^2&x_2(t)^2&x_3(t)^2& \ldots &x_n(t)^2\end{bmatrix}.$$ We have $\det(M'(i,j,k))=(x_j(t)-x_i(t))(x_k(t)-x_i(t))(x_k(t)-x_j(t))$, hence $$D'(i,j)+D'(i,k)+D'(j,k)=\deg(\det(M'(i,j,k))).$$ 

Let $M$ be the matrix obtained from $M'$ by multiplying, for each $i$, the $i$-th column of $M'$ by $(t^{D(i,n)-E})^2$. Since 
\begin{eqnarray*} D(i,j) &=&D'(i,j)+(D(i,n)-E)+(D(j,n)-E)\\ &=&\deg\left(t^{D(i,n)-E} \cdot t^{D(j,n)-E} \cdot (x_i(t)-x_j(t))\right),\end{eqnarray*} 
we get that $D(i,j)+D(i,k)+D(j,k)=\deg(\det(M(i,j,k)))$. If we replace each $t$ in $M$ by $t^{-1/2}$, we get $$D(i,j,k)=-\val(\det(M(i,j,k))).$$

Now assume $T$ has irrational edge weights. We can approximate $T$ arbitrarily close by a tree $\widetilde{T}$ with rational edge weights. From the arguments above, it follows that the $3$-dissimilarity map $\widetilde{D}$ of $\widetilde{T}$ belongs to $S$, hence $D\in \mathcal{G}_{3,n}$.   
\end{proof}

\begin{Thm} \label{thm m=3} If $n\geq 5$, we have $\phi^{(3)}(\GT)=\phi^{(3)}(\RR^{n\choose 3})\cap \mathcal{G}_{3,n}$.
\end{Thm}
\begin{proof}
The statement follows from Proposition \ref{Prop equality for m=3}, Proposition \ref{Thm inclusion} and the fact that $\mathcal{G}_{3,n}\subset \mathcal{T}_{3,n}$.
\end{proof}

\section{The $4$-dissimilarity maps of trees} \label{section m=4}

In this section, we give a geometric description of $\phi^{(4)}(\GT)$.

\begin{Rmk} {\fontshape{n}\selectfont
The set $\mathcal{C}_4=\{\sigma_1,\sigma_1^{-1},\sigma_2,\sigma_2^{-1},\sigma_3,\sigma_3^{-1}\}$ with $$\sigma_1=\begin{pmatrix} 1 & 2 & 3 & 4\\ 2& 3& 4 &1\end{pmatrix}, \sigma_2=\begin{pmatrix} 1 & 2 & 3 & 4\\ 2& 4& 1 &3\end{pmatrix}, \sigma_3=\begin{pmatrix} 1 & 2 & 3 & 4\\ 3& 4& 2 &1\end{pmatrix}.$$
Hence the map $\phi^{(4)}$ sends $(X(i,j))_{i,j}$ to $(X(i,j,k,l))_{i,j,k,l}$ where $X(i,j,k,l)$ is equal to the minimum of the three terms
\begin{align*} X(1,2)+X(2,3)+X(3,4)+X(4,1), \\X(1,2)+X(2,4)+X(4,3)+X(3,1),\\ X(1,3)+X(3,2)+X(2,4)+X(4,1),\end{align*} divided by two.
}\end{Rmk}

Consider $M=\RR^{({n \choose 2}\cdot {n-2\choose 2})}$ and take $X(i,j;k,l)$, with $\{i,j,k,l\}\subset[n]$ a quadruple, as coordinates on $M$. For example, $X(j,i;l,k) = X(i,j;k,l)$, but $X(i,k;j,l)\neq X(i,j;k,l)$ and $X(k,l;i,j)\neq X(i,j;k,l)$. 

Let $\pi:\RR^{n \choose 2}\to M: (X(i,j))_{i,j} \mapsto (X(i,j;k,l))_{i,j,k,l}$ with
$$X(i,j;k,l)= \frac12 \cdot(X(i,j)+X(k,l)+\min\{X(i,k)+X(j,l),X(i,l)+X(j,k)\}).$$ 

Let $L$ be the linear subspace of $M$ consisting of points $X(i,j;k,l)$ with
$$X(i,j;k,l)=X(i,k;j,l)=X(i,l;j,k)=X(j,l;i,k)=X(j,k;i,l)=X(k,l;i,j)$$ for all
different $i,j,k,l\in[n]$. Points in $L$ can be projected naturally to $\RR^{n \choose 4}$ by
sending $X(i,j;k,l)$ to $X(i,j,k,l)$. Denote this projection by $p$.

\begin{Prop} $\phi^{(4)}(\GT) = p(\pi(\RR^{n \choose 2}) \cap L )$.
\end{Prop}
\begin{proof}
Note that for any real numbers $a,b,c$, we have
\begin{equation}\label{eqn min} a+\min\{b,c\}=b+\min\{a,c\}=c+\min\{a,b\} \end{equation} if and only if
$\max\{a,b,c\}$ is attained at least twice. If the latter holds, the terms in \eqref{eqn min} are equal to $\min\{a+b,a+c,b+c\}$. 

If we take $a=X(i,j)+X(k,l)$, $b=X(i,k)+X(j,l)$ and $c=X(i,l)+X(j,k)$, the statement follows from the Tree Metric Theorem.
\end{proof}

\section*{Aknowledgments}
We thank Ruriko Yoshida, Anders Jensen and expecially Bernd Sturmfels for their many comments and suggestions which improved this manuscript. The second author is a postdoctoral fellow of the Research Foundation - Flanders (FWO).

\end{document}